\documentclass[11pt,reqno]{article}

\usepackage[usenames]{color}
\usepackage{graphicx}
\usepackage{multirow}
\usepackage{fullpage}
\usepackage{float}

\usepackage{amsmath}
\usepackage{amssymb}
\usepackage{amsthm}

\usepackage[colorlinks=true,
linkcolor=webgreen,
filecolor=webbrown,
citecolor=webgreen]{hyperref}

\definecolor{webgreen}{rgb}{0,.5,0}
\definecolor{webbrown}{rgb}{.6,0,0}

\theoremstyle{plain}
\newtheorem{theorem}{Theorem}[section]
\newtheorem{corollary}[theorem]{Corollary}
\newtheorem{conjecture}[theorem]{Conjecture}

\newtheorem{proposition}[theorem]{Proposition}

\numberwithin{equation}{section}
\numberwithin{theorem}{section}
\numberwithin{table}{section}
\numberwithin{figure}{section}

\title{Spot-Based Generations for Meta-Fibonacci Sequences}
\author{Barnaby Dalton \quad Mustazee Rahman \quad Stephen Tanny\\
\small Department of Mathematics\\
\small University of Toronto\\
\small Toronto, Ontario M5S 2E4\\
\small Canada\\
\texttt{bdalton@ca.ibm.com}\,, \texttt{mustazee.rahman@utoronto.ca}\,, \texttt{tanny@math.utoronto.ca}
}
\date{June 14, 2010}

\begin{document}

\maketitle

\begin{abstract}
For many meta-Fibonacci sequences it is possible to identify a partition of the sequence into successive intervals (sometimes called blocks) with the property that the sequence behaves ``similarly" in each block. This partition provides insights into the sequence properties. To date, for any given sequence, only ad hoc methods have been available to identify this partition. We apply a new concept - the spot-based generation sequence - to derive a general methodology for identifying this partition for a large class of meta-Fibonacci sequences. This class includes the Conolly and Conway sequences and many of their well-behaved variants, and even some highly chaotic sequences, such as Hofstadter's famous $Q$-sequence.
\end{abstract}

\section{Introduction} \label{Sec:intro}

In this paper we explore certain properties of the solutions to the recursions in two very general families, both of which have received increasing attention of late (see, for example, \cite{Grytczuk3, Rpaper, FRusCDeu} and the references cited therein). The first of these recursion families is defined as follows: for the positive integer $k>1$ and nonnegative integer parameters $a_p, b_p$, $p = 1\ldots k$,

\begin{equation}\label{introeqn1}
C(n)=\sum_{p=1}^k C(n-a_p-C(n-b_p)).
\end{equation} We often abbreviate a recursion in this family by $(a_1, b_1, a_2,b_2,..., a_k, b_k)$. The second family of recursions is defined by

\begin{equation}\label{introeqn2}
A(n)=A(n-A^k(n-1))+A(A^k(n-1))
\end{equation} where $k>0$ and $A^k(n)$ means a $k$-fold composition of the function $A$. For convenience, in our notation for both recursion families we suppress the parameter $k$ from the variable name. It will be evident from the context when a specific value of $k$ is intended.

Recursions $(\ref{introeqn1})$ and $(\ref{introeqn2})$ are examples of so-called meta-Fibonacci recursions, which refers to the ``self-referencing" nature of these recursions. An integer sequence is a meta-Fibonacci sequence if it is a solution to a meta-Fibonacci recursion.
Many well-known meta-Fibonacci recursions, with specified initial conditions, are special cases of the above two recursion families. Examples of $(\ref{introeqn1})$ include: Hofstadter's $Q$-recursion $(0,1,0,2)$ \cite{GEB, Guy}, the Conolly recursion $(0,1,1,2)$ \cite{Conolly, Tanny}, and the celebrated V-recursion $(0,1,0,4)$ \cite{BKT}. Two special cases of $(\ref{introeqn2})$ have been examined in detail. For $k=1$ this is the meta-Fibonacci recursion variously attributed to Conway, Hofstadter and Newman (see \cite{Vakil_Tal, Mallows, Newman} for more on this), while the case $k=2$ is explored in \cite{Grytczuk3}.

For the last four of these examples (that is, excluding the $Q$-recursion), the solution with initial conditions all set to 1 is completely understood. In particular, each of the resulting sequences is monotonically increasing, with the difference between successive terms always 0 or 1. Following Ruskey, we call such a sequence slow-growing, or slow. For each of these meta-Fibonacci sequences, and indeed for many others (including the $Q$-sequence), it is possible to identify a partition of the domain of the sequence into successive intervals (sometimes called blocks) with the property that the sequence behaves roughly ``in the same way" in each block. See, for example, \cite{Conolly, Mallows, Tanny, BKT}, where the nature of the block structure has been characterized precisely for the slow-growing sequences mentioned above. In each case, the approach to identifying this partition and what is meant precisely by behaves ``in the same way" varies from one sequence to the next; in all cases, however, the basic idea is that there appears to be a discernible pattern in the behavior of the sequence that repeats in successive blocks. This property can also be found in meta-Fibonacci sequences with much more chaotic behavior; in \cite{Pinn}, Pinn provides considerable experimental evidence for the existence of an underlying block structure in Hofstadter's $Q$-sequence.

In this paper we introduce an approach that formalizes and unifies this heuristic notion of an underlying block structure for a meta-Fibonacci sequence that is a solution to recursion $(\ref{introeqn1})$ or $(\ref{introeqn2})$. In so doing we explicitly connect the block structure to the form of the recursion and its parameters in an intuitive way. As a result, for an arbitrary sequence defined by these recursions, we identify a partition that often appears to highlight important properties of the sequence for further consideration. Such insight into the apparent block structure of a yet unknown sequence can provide helpful guideposts for developing conjectures and proofs.\footnote{See, for example, \cite{CCT}, where block structure insights are used to help identify and formulate the appropriate approach and specific induction assumptions required to prove the behavior of a family of sequences related to (\ref{introeqn1}).}

\section{Spot-Based Generations} \label{Sec:defs}

Define a homogeneous meta-Fibonacci recursion to be any recursion of the form:
\begin{equation} \label{eqnMFib}
T(n) = \displaystyle \sum_{p=1}^k T(S_p(n, T_{<n}))\
\end{equation}
We refer to the function $S_p(n, T_{<n})$ as the $p^{th}$ spot function; it depends on the index $n$ and values of $T(j)$ for $j < n$, which we indicate by the symbol $T_{<n}$. In the homogeneous recursion (\ref{introeqn1}), the spot functions are $S_p(n, C_{<n}) = n-a_p-C(n-b_p)$ for $1\leq p\leq k$. In the homogeneous recursion (\ref{introeqn2}) there are two spot functions, namely, $S_1(n, A_{<n}) = n-A^k(n-1)$ and $S_2(n,A_{<n}) = A^k(n-1)$. For convenience and when the context is clear, we often write $S_p(n)$ in place of $S_p(n, T_{<n})$.

To ensure that $T(n)$ is defined by (\ref{eqnMFib}) for all $n$, we require that for $1\leq p\leq k$, $S_p(n)<n$ for all $n$ following the initial conditions. We assume this holds for the recursions that we discuss.\footnote{If it fails then for the smallest integer $n$ for which it fails, we say that the sequence terminates at index $n$.} For each spot function $S_p(n)$, we define a new sequence by the recursion
\begin{equation} \label{eqnspot}
M_p(n) = M_p(S_p(n)) + 1\
\end{equation}
with initial conditions $M_p(n) = 1$ for $n = 1, \ldots, r$ where $r$ is to be the same as the number of initial conditions used in the definition of $T(n)$.\footnote{In general this value of $r$ will be greater than the minimum value that is required by the specific nature of the recursion (\ref{eqnspot}); further, this minimum value can differ for different values of $p$.}

We call the sequence $M_p(n)$ the generation sequence for $T(n)$ based on spot $p$.\footnote{We often refer to $M_p(n)$ as the generation structure for $T(n)$ based on spot $p$, especially when we are considering the overall characteristics of this sequence rather than the behavior of individual terms.} For $g\geq 1$ we define the $g^{th}$ generation with respect to the $p^{th}$ spot function to be the set $M_p^{-1}(\{g\})$, which we denote $G_p(g)$. For ease of notation we may omit reference to the index $p$ in $G_p(g)$ when the index $p$ is clear from the context. The definition of $M_p(n)$ is motivated by considering index $n$ to be in the ``next generation" of its $p^{th}$ spot ancestor $S_p(n)$, which itself is a member of a previous generation with generation number $M_p(S_p(n))$. When there are two spot functions we call $M_1(n)$ the mother function and $M_2(n)$ the father function (here we adapt terminology introduced by Pinn \cite{Pinn}). We call the generation sequences that result from these spot functions the maternal and paternal generation sequences, respectively. Similarly, the $g^{th}$ generations in this case are called the $g^{th}$ maternal and paternal generation, respectively.

It follows immediately from the recursion (\ref{eqnspot}) for $M_p$ and the initial conditions that the generation sequence begins at 1 and is onto either all of the positive integers or an interval of the positive integers beginning at 1. For fixed $p$, it is often the case that the $g^{th}$ generation $G_p(g)$ is a finite interval of positive integers for all $g \geq 1$, and the generations partition the positive integers into intervals. However, this is not always the case. We discuss this, together with a variety of other issues, in the following sections where we apply our generation notion to specific meta-Fibonacci sequences.

At this point an example may be helpful. In the notation we introduced above, the Conolly sequence $(0,1,1,2)$ \cite{Conolly} is given by $C(n) = C(n-C(n-1)) + C(n-1-C(n-2))$, with initial conditions $C(1) = C(2) = 1$. It is well-known that $C(n)$ is slow, and that for each $n$, $C(n)$ equals $n$ exactly $\nu_2(2n)$ times, where $\nu_2(n)$ is the highest power of 2 that divides $n$. The behavior of $C(n)$ between successive powers of 2 provided important insights for formulating the original induction-based proofs of the properties of $C(n)$ (see \cite{Tanny}).

The maternal generation sequence for $C(n)$ is given as $M_1(n) = M_1(n-C(n-1)) + 1$ with $M_1(1) = M_1(2) = 1$.
In the next section we prove that $M_1(n)$, the maternal generation sequence of the Conolly sequence, is slow-growing, and further, that $G_1(g) = [2^{g-1}+1, 2^g]$ for all $g \geq 2$. In this case the generation structure aligns at successive powers of $2$, exactly where the natural division points for the ``frequency" function $\nu_2(2n)$ of $C(n)$ are observed to occur.

For any homogeneous meta-Fibonacci recursion, define the beginning of the $g^{th}$ generation with respect to $M_p$ by $\alpha_p(g) = \min \{n \,|\, n \in G_p(g)\}$. Similarly define the end of the $g^{th}$ generation with respect to $M_p$ by $\beta_p(m) = \max \{n \,|\, n \in G_p(g)\}$ provided it exists. When the context is clear we will drop the subscript $p$ from the notation. By definition $G_p(g) \subseteq [\alpha_p(g), \beta_p(g)]$. If $G_p(g) \neq [\alpha_p(g), \beta_p(g)]$, we say that the $g^{th}$ generation $G_p(g)$ is \textbf{fragmented}. Otherwise when $G_p(g) = [\alpha_p(g), \beta_p(g)]$ for all $g \geq 1$, we say that the generational structure with respect to $M_p$ has an \textbf{interval structure}. This is the case for the Conolly sequence above.

\section{Generation Sequences Based On Slow-Growing Spot Sequences}\label{Sec:slowGrowing}

For the recursions (\ref{introeqn1}) and (\ref{introeqn2}) the spot sequences are of the from $S(n) = n-a-T^k(n-b)$ or $S(n) = T^k(n-b)$, respectively. These spot sequences will be slow-growing if the original sequence itself is slow-growing. For this reason we turn our attention to the situation where the spot sequence $S_p(n)$ in ~(\ref{eqnMFib}) is slow-growing. In this case much can be deduced about the generational structure of $T(n)$ based on spot $p$.

\begin{theorem} \label{thm1}
For the meta-Fibonacci sequence $T(n)$ in ~(\ref{eqnMFib}), if the spot function $S_p(n)$ is slow-growing, then so is the $p^{th}$ spot-based generation sequence $M_p(n)$.
\end{theorem}

\begin{proof} The proof is by induction on $n$. For the base case, note that by the initial conditions in (\ref{eqnspot}),
$M(n) = 1$ for $ 1 \leq n \leq r$. Let $\Delta M(n) = M(n+1) - M(n)$; note that if
$r > 1$, then $\Delta M(1) = 0$. If $r=1$ then for $T(2)$ to be well-defined we must have $S_p(2) = 1$, from which it follows that $\Delta M(1) = M(2) - M(1) = M(S_p(2)) - M(S_p(1)) = M(1)-M(1)=0$. Thus, in all cases, $\Delta M(1) \in \{0,1\}$.

For $n > 1$, assume that $\Delta M(k) \in \{0,1\}$ for all $k < n$. Since $S_p(i)$ is slow-growing, observe that $M(n+1) = M(S_p(n+1)) + 1 = M(S_p(n) + t) + 1$ where $t = 0$ or $t = 1$. Since $T(n)$ is well-defined, we must have that $S_p(n) + t \leq n$. Thus, by the induction assumption, $M(S_p(n) + t) = M(S_p(n)) + j$ where $j \in \{0,1\}$. It follows that $M(n+1) = [M(S_p(n)) + 1] + j = M(n) + j$ from which we get that $\Delta M(n) = j \in \{0,1\}$. This completes the induction.
\end{proof}

\begin{corollary} \label{cor}
If $S_p(n)$ is slow-growing, then the generation sequence of $T(n)$ based on spot $p$ is an interval structure.
Further, in this case, for $g \geq 1$, $\beta_p(g) = \alpha_p(g+1) - 1$.
\end{corollary}

\begin{proof} As $M_p(n)$ is slow-growing, for any $g \geq 1$ there is a minimum index $\alpha_p(g)$ and a maximum index $\beta_p(g)$ such that $M_p(\alpha_p(g)) = M_p(\beta_p(g)) = g$. If $S_p(n)$ becomes constant then for some $g*$ we have $\beta_p(g*) = \infty$, and thus there are only a finite number of generations. Otherwise, for every $g \geq 1$, $G_p(g) = [\alpha_p(g), \beta_p(g)]$ and $\beta_p(g) = \alpha_p(g+1) - 1$.
\end{proof}

For a slow-growing spot sequence $S_p(n)$, there is an elegant relation between this spot sequence and the generation sequence $M_p(n)$ that it induces. By the definition of $\alpha_p(g+1)$, we have that $g+1 = M_p(\alpha_p(g+1)) = M_p(S_p(\alpha_p(g+1))) + 1$. So, $S_p(\alpha_p(g+1)) \geq \alpha_p(g)$ since $\alpha_p(g)$ is the beginning of generation $g$. To see that equality holds, assume the contrary, namely, $S_p(\alpha_p(g+1)) > \alpha_p(g)$. Since $S_p(n)$ is slow-growing, there exists $\gamma < \alpha_p(g+1)$ such that $S_p(\gamma) = \alpha_p(g)$. But this implies that $M_p(\gamma) = M_p(S_p(\gamma))+1 = M_p(\alpha_p(g)) + 1 = g+1$, which contradicts the definition of $\alpha_p(g+1)$. Thus, $S_p(\alpha_p(g+1)) = \alpha_p(g)$.

Since $S_p(n)$ is slow, $M_p(n)$ is also slow. Thus, for every $g \geq1,\; \beta_p(g) = \alpha_p(g+1) - 1$. By what we have just shown $S_p(\alpha_p(g+2)) = \alpha_p(g+1)$, so using the fact that $S_p(n)$ is slow, we have that $S_p(\beta_p(g+1))\in \{\alpha_p(g+1) - 1, \alpha_p(g+1)\}$. However, $S_p(\beta_p(g+1)) = \alpha_p(g+1)$ would imply that $M_p(\beta_p(g+1)) = M_p(\alpha_p(g+1)) + 1 = g+2$. This is a contradiction to the definition of $\beta_p(g+1)$. Thus, $S_p(\beta_p(g+1)) = \alpha_p(g+1) - 1 = \beta_p(g)$, and we have proved the following result:

\begin{theorem} \label{thm2}
Suppose that $S_p(n)$ is a slow-growing spot of $T(n)$. Then for every $g>0$, $S_p(n)$ maps the $(g+1)^{th}$ generation onto the $g^{th}$ generation. That is, $S_p(\alpha_p(g+1)) = \alpha_p(g)$ and $S_p(\beta_p(g+1)) = \beta_p(g)$.
\end{theorem}

Any slow-growing spot $S_p(n)$ of $T(n)$ specifies a generation structure that is uniquely determined by the sequence of generation interval starting points $\alpha_p(g)$ for all positive integers $g$. If the recursion for $T(n)$ has $r$ initial conditions, then the starting points of the generations are uniquely determined by the property that $\alpha_p(2) = r + 1$ and for all subsequent $\alpha_p(g)\; g > 2$, $\alpha_p(g)$ is the \emph{smallest} number with the property that $S_p(\alpha_p(g+1)) = \alpha_p(g)$. We use this fact extensively in what follows, where we compute the maternal generation structure for several well-known slow-growing sequences. In so doing we show how our spot-based maternal generations are essentially congruent to the block structures for these sequences that are identified in an ad hoc way in the literature.

We begin with the Conway sequence and several of its variants. The Conway sequence is defined by $A(n) = A(n-A(n-1)) + A(A(n-1))$, with $A(1) = A(2) = 1$. The graph of the Conway sequence in Figure \ref{fig:Conway} provides striking evidence of what is usually meant by a block structure: a series of arcs between successive powers of 2 indicating that the behavior of the sequence is essentially the same on these intervals of the domain.

\begin{figure}[htpb]
\begin{center}
\includegraphics[scale=0.7]{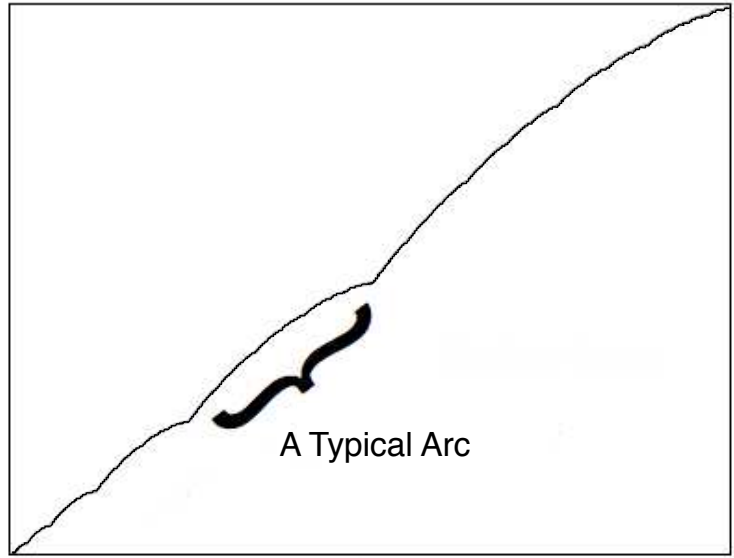}
\caption{Graph of Conway sequence $A(n)$ for $1 \leq n \leq 1024$.} \label{fig:Conway}
\end{center}
\end{figure}

The interval $[2^{m}, 2^{m+1})$ has been termed the $m^{th}$ ``octave" of the sequence (\cite{Conolly}). Various patterns in the Conway sequence have been shown to persist from one octave to the next. For example, $A(2^m) = 2^{m-1}$ for $m \geq 1$, that is, the beginning of each octave is mapped to the beginning of the previous one. Also, for $m \geq 2$, $A(n) = 2^{m-1}$ for exactly the last $m$ indices in the $m^{th}$ octave, and $A(n) \geq \frac{n}{2}$ with equality only when $n$ is a power of 2.

It is readily confirmed that the maternal generation structure of $A(n)$ conforms naturally to these octaves.

\begin{table}[!htp]
\fontsize{9}{9}\selectfont
\caption{First 10 maternal generations of the Conway Sequence.} \label{tbl:gen_Conway1}
\center{
\begin{tabular}{|c | c c c c c c c c c c|}\hline
$g$&1&2&3&4&5&6&7&8&9&10\\\hline
$G_1(g)$&[1, 2]&[3, 4]&[5, 8]&[9, 16]&[17, 32]&[33, 64]&[65, 128]&[129, 256]&[257, 512]&[513, 1024]\\\hline
\end{tabular}
}
\end{table}

\begin{proposition}
For the Conway sequence, $\alpha_1(g) = 2^{g-1} + 1$ for $g > 1$.
\end{proposition}

\begin{proof}We proceed by induction on $g$. The base case is clear from Table \ref{tbl:gen_Conway1}. We use the fact that for $g>0$, $A(2^g) = A(2^g - 1) = 2^{g-1}$ (\cite{Vakil_Tal}). Assume the proposition up to generation $g$. For $g+1$ we have that $M(2^g + 1) = M(2^g + 1 - A(2^g)) + 1$. As $A(2^g) = 2^{g-1}$,
we get $2^g + 1 - A(2^g) = 2^{g-1} +1$. By the induction hypothesis we have that $\alpha_1(g) = 2^{g-1} + 1$, and so $M(2^g + 1) = M(\alpha_1(g)) + 1 = g+1$. On the other hand, $M(2^g) = M(2^g - A(2^g-1)) + 1$, and since $A(2^g - 1) = 2^{g-1}$, this implies $2^g - A(2^g-1) = 2^{g-1} = \alpha_1(g) - 1$. Hence, $M(2^g) = M(\alpha_1(g) - 1) + 1 = g-1 +1 = g$. Since $M(n)$ is slow, it follows that $M(2^g + 1)$ is the first occurrence of $g+1$, so $\alpha_1(g+1) = 2^g + 1$. This completes the induction. \end{proof}

Next we show that an analogous result holds for the Newman-Conway sequences (see \cite{Newman}). These sequences are defined by $f_r(n) = f_r(n-f_r(n-1)) + f_r(f_r(n-1))$, with initial conditions $f_r(i) = 1$ for $1 \leq i \leq r+1$ and $r \geq 1$. Note that the Conway sequence corresponds to $r=1$. For any fixed $r>1$, the following holds: (1) the sequence $f_r(n)$ is slow-growing; (2) like that of the Conway sequence, the graph of $f_r(n)$ consists of successive arcs that begin and end at the ``Fibonacci-type" numbers $E_n$ defined by $E_n = E_{n-1} + E_{n-r}$, with initial conditions $E_i = 1$ for $1 \leq i \leq r$; (3) these arcs identify a natural partition of the domain at the points $E_n$; (4) for $n>r, f_r(E_n) = f_r(E_n - 1) = E_{n-r}$ (see (\cite{Kleitman})).

As is the case for the Conway sequence, the maternal generations for the Newman-Conway sequences create essentially this same partition of the domain. More precisely, for $g > 1$, the maternal generation begins at $\alpha_1(g) = E_{2r + g - 2} + 1$. Clearly this holds for the second generation, which starts at $r + 2 = E_{2r} + 1$. By definition, $S_1(E_n + 1) = E_n + 1 - f_r(E_n)$. Since $f_r(E_n) = E_{n-r}$ for $n > r$, we have that $S_1(E_n + 1) = E_n - E_{n-r} + 1 = E_{n-1} + 1$ by the recursion for $E_n$. But $f_r(E_n - 1) = E_{n-r}$ for $n > r$, so we have that $S_1(E_n) = E_n - f_r(E_n - 1) = E_n - E_{n-r} = E_{n-1}$. By Theorem \ref{thm2}, $S(\alpha_1(g+1)) = \alpha_1(g)$ and $\alpha_1(g+1)$ is the smallest number with this property. Since $\alpha_1(2) = E_{2r} + 1,\;S(E_n) = E_{n-1}$ and $S(E_n +1) > E_n$, we can use induction together with the discussion following Theorem \ref{thm2} to deduce that $\alpha_1(g) = E_{2r + g -2} + 1$.

We conclude our consideration of Conway sequence variants with the sequences defined by (\ref{introeqn2}) for $k>1$, and with initial conditions $A(1) = A(2) = 1$ (see \cite{Grytczuk3}). For $k=2$, Grytczuk proved that the resulting sequence is slow-growing. He also showed that the last occurrence of the Fibonacci number $F_n$ in the sequence occurs at position $F_{n+1}$. This prompts Grytczuk to observe that this sequence has a clearly identifiable block structure, in which the Fibonacci numbers play a prominent role. He states: ``this suggests to divide $A(n)$ into segments of the form $[A(F_n+1), A(F_n+2),...,A(F_{n+1})]$" \cite[p.\ 149]{Grytczuk3}.
Once again, just as we found for the Conway and Newman-Conway sequences, the maternal generation structure for Grytczuk's sequence matches the natural block structure that he identified. In fact, even more is true: an analogous result holds for $(\ref{introeqn2})$ with any positive $k>1$ and the same initial conditions $A(1) = A(2) = 1$.

We now outline the argument for this, which follows the same lines as the one given above\footnote{The interested reader can contact us for additional details.}. For any fixed $k>1$ the sequence $A(n)$ defined by (\ref{introeqn2}) is slow-growing. For $n>k$, define $E_n$ by $E_n = E_{n-1} + E_{n-k}$, with initial conditions $E_i = 1$ for $1 \leq i \leq k$. Then for $n > k$, we can show that $A(E_{n+1}) = E_n$, $E_{n+1}$ marks the last occurrence of the value $E_n$, and $A^k(E_n-1) = E_{n-k}$. It follows that $A(n)$ has a natural block structure whose division points are at the points $E_n$.

Using these properties we will show that the maternal generations of $A(n)$ based on the spot $n - A^k(n-1)$ have an interval structure, where for $g > 1$ the $g^{th}$ generation begins at $\alpha_1(g) = E_{k+g-1} + 1$. As such, the maternal generation structure coincides with the block structure for $A(n)$ that we just described.
To see this, note first that $\alpha_1(2) = 3 = E_{k+1} + 1$. For $g > 2$, by using the aforementioned properties of $A(n)$, we get
$$S_1(E_{k+g-1} + 1) = E_{k+g-1} + 1 - A^k(E_{k+g-1}) = E_{k+g-1} - E_{g-1} + 1 = E_{k+g-2} + 1$$
Also, we have that
$$S_1(E_{k+g-1}) = E_{k+g-1} - A^k(E_{k+g-1} - 1) = E_{k+g-1} - E_{g-1} = E_{k+g-2}$$
By the remarks immediately following Theorem \ref{thm2}, these properties uniquely determine the start points for the maternal generations. Thus, for $g>2$, we have that $\alpha_1(g) = E_{k+g-1}$.

In our final example we show, as asserted in Section \ref{Sec:defs}, that the maternal generation sequence of the Conolly sequence $C(n)$ discussed in Section \ref{Sec:defs} is slow-growing. Further, for $g>1$, $G_1(g) = [2^{g-1}+1, 2^g]$, that is, the $g^{th}$ maternal generation is a shift of $1$ from the $g^{th}$ block $[2^{g-1}, 2^g -1]$ identified by Conolly \cite{Conolly}.

We want to show that for $g>1$ we have $\alpha_1(g) = 2^{g-1} + 1$. Similar to our earlier arguments, since $C(n) = C(n-C(n-1)) + C(n-1-C(n-2))$, it suffices to verify that $S_1(2^{g-1} + 1) = 2^{g-2} + 1$ and $S_1(2^{g-1}) = 2^{g-2}$. Since $S_1(n) = n-C(n-1)$, these are equivalent to $C(2^g) = C(2^g - 1) = 2^{g-1}$, which is a well known property of the Conolly sequence (see, for example, \cite{Tanny}). This completes the proof.

\section{Generational Structure For Selected Non-Slow Sequences} \label{Sec:nonslow}
Based on our initial experimental evidence, we believe that an analysis of generation structures can provide useful insights for sequences with more erratic behavior than that of the slow-growing sequences addressed in the previous section. For example, a sequence of interest is $\mu(n)$ generated by the recursion $(1,2,2,1)$ with three initial conditions all equal to 1:

\begin{equation} \label{seq:mu}
\mu(n) = \mu(n-1-\mu(n-2)) + \mu(n-2-\mu(n-1)); \quad \mu(1) = \mu(2) = \mu(3) = 1
\end{equation}

Table \ref{tbl:mu} contains the first 50 values of $\mu(n)$. The sequence $\mu(n)$ is neither slow-growing nor monotonic. It is not even known whether $\mu(n)$ is defined for all positive integers $n$, that is, whether $n-2-\mu(n-1) > 0$ for all integers $n \geq 4$.

\begin{table}[!htpb]
\fontsize{10}{10}\selectfont
\caption{First 50 terms of $\mu(n)$}\label{tbl:mu}
\center{
\begin{tabular}[t]{ |l| r r r r r r r r r r|}\hline
\multirow{2}{*}{}&&&&&$n$&&&&&\\
&1&2&3&4&5&6&7&8&9&10\\\hline
$\mu(n+0)$&1&1&1&2&2&2&3&3&4&4\\
$\mu(n+10)$&4&5&5&6&7&7&7&8&8&8\\
$\mu(n+20)$&9&9&10&11&11&11&13&12&14&13\\
$\mu(n+30)$&14&15&15&15&16&16&16&17&17&18\\
$\mu(n+40)$&19&19&19&21&20&22&21&22&24&24\\\hline
\end{tabular}
}
\end{table}

At the same time, an examination of the first $10^6$ terms of $\mu(n)$ indicates that the sequence appears to have many regularities. For example, $\mu(n)$ hits every power of $2$ exactly three times, and for each power of $2$ the three occurrences are for consecutive arguments. Each of these runs of $2^k$ is also preceded by at least two consecutive occurrences of $2^k-1$ and succeeded by precisely two occurrences of $2^k+1$ (see Table \ref{tbl:mu} for examples of this behavior). Additional regularities are evident in the graph of $\mu(n)$: the graph alternately widens and narrows, and the general appearance on each interval, defined by successive narrowing of the sequence, is similar (see Figure \ref{fig:mu}). The narrowing of $\mu(n)$ occurs where the sequence takes on the values of a power of 2. The lengths of these successive intervals doubles as does the amplitude of the variation about the trend line of the graph.

\begin{figure}[!htb]
\begin{center}
\includegraphics[scale=.25]{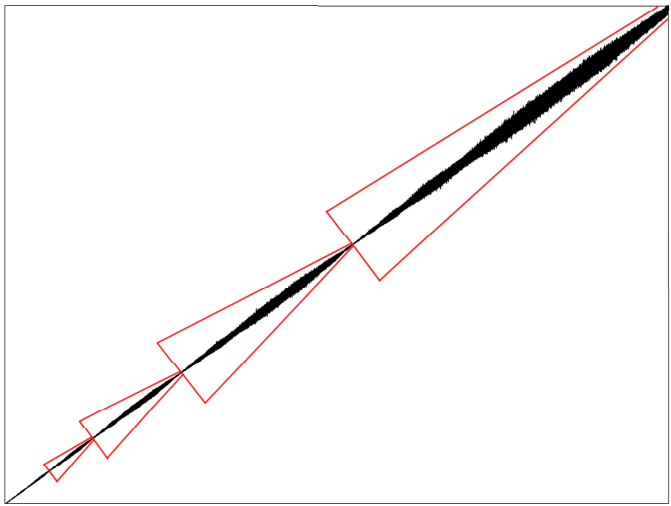}
\caption{Graph of $\mu(n)$ for $1 \leq n \leq 10^6$.} \label{fig:mu}
\end{center}
\end{figure}

It is fascinating that the maternal generation structure for the first $10^6$ terms of $\mu(n)$ corresponds precisely with the intervals identified from the graph. Based on our experimental evidence we conjecture the following:
\begin{conjecture}The sequence $\mu(n)$ is defined for all positive integers $n$. The maternal generation sequence of $\mu(n)$ is slow-growing.
For each $g \geq 3$ the $g^{th}$ maternal generation begins at index $2^{g-1} + g$, which is the first occurrence of $2^{g-2} + 1$ in $\mu(n)$, and ends
at index $2^g + g$, which is the last occurrence of $2^{g-1}$ in $\mu(n)$.\end{conjecture}

We conclude by describing some intriguing experimental findings concerning the maternal generation structure for Hofstadter's famous $Q$-sequence (0,1,0,2) \cite{GEB}. First we set the stage. It is well known (for example, see \cite{Pinn}) that $Q(n)$ exhibits the following repetitive behavior about its underlying trend line $y=n/2$: a period of relatively large oscillations, sometimes initiated by a large ``spike", followed by gradually decreasing oscillation that tapers to a portion of relative quiet with very small differences, and then the process repeats (see Figure \ref{fig:Hof}). In \cite{Pinn} Pinn locates the first 20 ``transition points" where the large oscillations in $Q(n)-n/2$ recur following a period of relative calm (see the right hand side of Table \ref{tbl:Pinn_mat_comp}). He identifies these points with the start points of the intervals that partition the domain into an underlying block structure for the $Q$-sequence (he terms these blocks ``generations"). Pinn finds the first eleven of these transition points ``by eye" from the graph of $Q(n)$.\footnote{The transition points double for generations 2 through 10, as does the value of $Q(n)$ at each of these 9 points, these values being the first occurrence of $2,\ldots, 2^9$ respectively. This pattern fails for the start point of the $11^{th}$ Pinn generation.} For subsequent generations, Pinn observes that ``the onset of the new generations is a little less well defined"\cite[p.\ 8]{Pinn}. By applying certain statistical approaches, Pinn concludes that a good estimate for the start point for generation $g>11$ is $\lfloor 2^{g-1/2} \rfloor$.\footnote{Observe the typographical error in \cite[p.\ 8]{Pinn}, where he writes $\lfloor 2^{g+1/2} \rfloor$.}

\begin{figure}[!hbtp]
\begin{center}
\includegraphics[scale=.5]{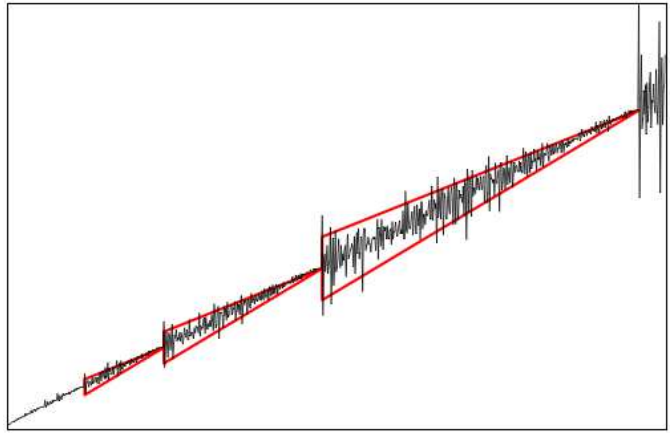}
\caption{Graph of $Q(n)$ for $1 \leq n \leq 800$.} \label{fig:Hof}
\end{center}
\end{figure}

As we now show, our spot-based maternal generation structure for $Q(n)$ does a much better job than Pinn's statistical estimation techniques at identifying the locations of the transitions in $Q(n)$, at least for all the first 18 generations that we have checked. Let $\alpha(g,\pi)$ denote the start point of the $g^{th}$ Pinn generation, while $\alpha_1(g)$ denotes the start point of the $g^{th}$ maternal generation of $Q(n)$ based on spot $n-Q(n-1)$. In Table \ref{tbl:Pinn_mat_comp} we compare the start points for the first 20 maternal generations of $Q(n)$ with those for the first 20 Pinn generations.

\begin{table}[!ht]
\fontsize{10}{10}\selectfont
\caption{Comparison of start points for maternal generations with Pinn's generations.}\label{tbl:Pinn_mat_comp}
\center{
\begin{tabular}{lr r r r r r r r r r|}
    \hline
    {}&\multicolumn{4}{c}{maternal generation $g$}&{}&\multicolumn{4}{c}{Pinn generation $g$}\\
    \cline{2-5}\cline{7-10}
    {}&1&2&3&4&{}&1&2&3&4\\
    \hline
$\alpha_1(g+0)$&1&3&6&12&$\alpha(g+0; \pi)$&1&3&6&12\\
$\alpha_1(g+4)$&24&48&96&192&$\alpha(g+4; \pi)$&23&48&96&192\\
$\alpha_1(g+8)$&384&768&1522&3031&$\alpha(g+8; \pi)$&384&768&1522&2896\\
$\alpha_1(g+12)$&6043&12056&24086&48043&$\alpha(g+12; \pi)$&5792&11585&23170&46340\\
$\alpha_1(g+16)$&95286&189268&376996&750285&$\alpha(g+16; \pi)$&92681&185363&370727&741455\\
\hline
\end{tabular}
}
\end{table}

From Table \ref{tbl:Pinn_mat_comp} we see that the start points match for the first eleven maternal and Pinn generations, respectively. This is very surprising, since the two approaches for identifying these points are entirely unrelated. For the remaining generations there are substantial differences in the start points, with the maternal generation start point bigger in every case.

\begin{table}[!ht]
\fontsize{10}{10}\selectfont
\caption{Deviations in $Q(n)$ at start points of maternal and Pinn generations.  The last column shows
the transition points for $Q(n)-n/2$.}\label{tbl:Q-deviations}
\center{
\begin{tabular}{lc c c c c cl} \hline
$g$&$\alpha_1(g)$&Absolute&$\alpha(g,\pi)$&Absolute&Transition\\
&&deviation (\%)&&deviation (\%)&points\\\hline
12&3031&9.48&2896&0.68&3032\\
13&6043&1.52&5792&0.48&6042\\
14&12056&1.00&11585&0.22&12069\\
15&24086&5.72&23170&0.46&24064\\
16&48043&0.42&46340&1.00&48013\\
17&95286&2.60&92681&0.36&95182\\
18&189268&1.73&185363&0.28&189266\\\hline
\end{tabular}
}
\end{table}

In Table \ref{tbl:Q-deviations} we calculate the absolute percent deviation of $Q(\alpha_1(g))$ from its previous value $Q(\alpha_1(g) -1)$, that is, the value $\frac{|Q(\alpha_1(g)) - Q(\alpha_1(g) - 1)|}{Q(\alpha_1(g) - 1)} \times 100\%$, for maternal generations 12 to 18. We then compare that with the corresponding absolute deviation for $Q(\alpha(g, \pi))$. Note how these deviations are almost always higher for the maternal generation start point, suggesting that these points are more closely located to the upcoming transition point. In the last column of Table \ref{tbl:Q-deviations} we locate the actual transition points of $Q(n)$.\footnote{This is done via a careful examination of the behavior of the sequence $Q(n)-n/2$. The interested reader may contact us for details.} We observe that for each generation $g$ from 12 through 18, $\alpha_1(g)$ is much closer to the transition point for generation $g$ than $\alpha(g, \pi)$. For example, for $g=14$, $\alpha_1(14) = 12,056$, $\alpha(14,\pi)=11,585$, and we estimate the transition point to be $12,069$ (see Figure \ref{fig:4}; notice that Pinn's start point for generation 14 is located in the midst of a relatively quiet portion of the values of $Q(n)-n/2$, while the start point for maternal generation 14 is much closer to the upcoming spike in $Q(n)$ situated at the end of this quiet region).

\begin{figure}[!ht]
\begin{center}
\includegraphics[scale=0.5]{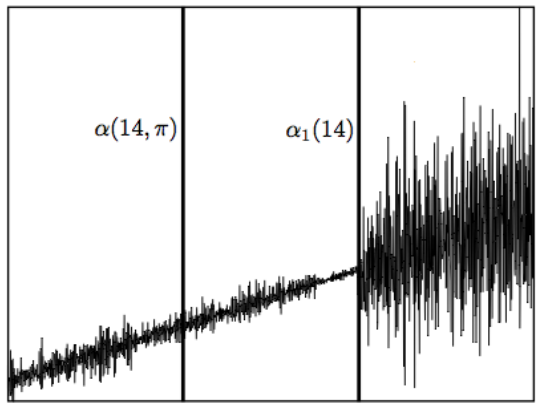}
\caption{Comparison of $\alpha_{1}(14)=12056$ and $\alpha(14;\pi)=11585$ on the graph of $Q(n)$} \label{fig:4}
\end{center}
\end{figure}

Further investigation is required to determine for how many generations beyond 18 this apparent connection between the start point of the maternal generations and the transition points for $Q(n)$ persists, and to understand its significance. But this tantalizing property of the maternal generation function for $Q(n)$ provides a further suggestion that a close examination of spot-based generation structures may provide potentially important insights into the intrinsic structure of many meta-Fibonacci sequences, including those with highly complex behavior such as $Q(n)$.

\section*{Acknowledgements} \label{ack}
We thank Brian Choi and Sahir Haider for their comments on the $\mu$ sequence. We would also like to acknowledge the computational assistance of Biao Zhou and Yin Xu on the $Q$-sequence.

\end{document}